\documentclass[a4paper,12pt]{article}
\usepackage{amsmath}
\usepackage{amssymb}
\usepackage{amsthm}
\usepackage{latexsym}
\usepackage{amsmath}
\usepackage{amsthm}
\usepackage[dvipdfmx]{graphicx}
\usepackage{txfonts}
\usepackage{bm}
\usepackage{color}
\usepackage{epic,eepic}

\usepackage{geometry}
\numberwithin{equation}{section}

\newtheorem{theorem}{Theorem}[section]
\newtheorem{proposition}[theorem]{Proposition}

\newtheorem{remark}{Remark}[section]
\newtheorem{example}{Example}[section]

\newcommand{\OMIT}[1]{{\bf [OMIT:} #1 \ {\bf --- end OMIT] }}  %%% For work
   \renewcommand{\OMIT}[1]{}            %%% For FINAL
\newcommand{\RR}{{\mathbb{R}}}
\newcommand{\ZZ}{{\mathbb{Z}}}
\newcommand{\vecone}{{\bf 1}}
\newcommand{\veczero}{{\bf 0}}
\newcommand{\dom}{{\rm dom\,}}

\newcommand{\unitvec}[1]{e\sp{#1}}

\newcommand{\conv}{\Box\,}

\newcommand{\finbox}{\hspace*{\fill}$\rule{0.2cm}{0.2cm}$}
\newcommand{\todaye}{\the\year/\the\month/\the\day}

\newcommand{\Lnat}{{L$^{\natural}$}}
\newcommand{\Mnat}{{M$^{\natural}$}}

\newcommand{\Rinf}{\overline{\RR}}

\newcommand{\Rmpinf}{\RR \cup \{ -\infty, +\infty\}}

\def\calF{{\cal F}}

\def\calT{{\cal T}}

\begin{document}

\title{On Fundamental Operations for \\  Multimodular Functions%
\thanks{
This work was supported by CREST, JST, Grant Number JPMJCR14D2, Japan, and
JSPS KAKENHI Grant Numbers 17K00037, 26280004.}
%% (Moriguchi=17K00037; Murota=26280004) 
}%title

\author{
Satoko Moriguchi%
\thanks{Department of Economics and Business Administration,
Tokyo Metropolitan University, 
%% Tokyo 192-0397, Japan, 
satoko5@tmu.ac.jp}
\ and 
Kazuo Murota%
\thanks{Department of Economics and Business Administration,
Tokyo Metropolitan University, 
%% Tokyo 192-0397, Japan, 
murota@tmu.ac.jp}
}%%author

%%\date{May, 2018 (Version \today)}
\date{May 6, 2018 / June 21, 2019}

\maketitle

\begin{abstract}
Multimodular functions, primarily used in the literature of 
queueing theory, discrete-event systems, and operations research, 
constitute a fundamental function class in discrete convex analysis.
The objective of this paper is to clarify the properties of multimodular functions
with respect to fundamental operations
such as permutation and scaling of variables, projection (partial minimization) and convolution.
It is shown, in particular, that the class of multimodular functions
is stable under projection under a certain natural condition 
on the variables to be minimized,
and the convolution of two multimodular functions
is not necessarily multimodular, even in the special
case of the convolution of a multimodular function
with a separable convex function.
\end{abstract}


%%\input{MMoperMainBody}
%%%2018-05-06
%%%2018-09-27
%%%2019-06-21

%%%\input{MMoperMainBody}
%%%2018-05-11
%%%2018-09-27
%%%2019-06-21
%% file =  MMoperMainBody.tex

\section{Introduction}
\label{SCintro}

Multimodular functions,
due to Hajek \cite{Haj85},
have been used as a fundamental tool 
in the literature of queueing theory, discrete-event systems, and operations research
\cite{AGH00,AGH03,FHS17,GY94mono,KS03,LY14,SW93markov,dWvS00,WS87netque,ZL10}.
In connection to discrete convex analysis \cite{Fuj05book,Mdca98,Mdcasiam,Mbonn09},
multimodularity can be regarded as a variant of \Lnat-convexity
in the sense that a function 
$f: \ZZ\sp{n} \to \RR \cup \{ +\infty\}$
is multimodular if and only if it can be represented as
$f(x) = g(x_{1}, \ x_{1}+x_{2}, \ \ldots, \  x_{1}+ \cdots + x_{n})$
for some \Lnat-convex function $g$ \cite{Mmult05}.

Various operations can be defined for discrete functions  
$f: \mathbb{Z}\sp{n} \to \mathbb{R} \cup \{ +\infty  \}$.
With changes of variables 
we can define operations such as 
an origin shift  $f(x) \mapsto f(x+b)$,
a sign inversion of variables $f(x) \mapsto f(-x)$, 
a permutation of variables 
 $f(x) \mapsto f(x_{\sigma(1)}, x_{\sigma(2)}, \ldots, x_{\sigma(n)})$,
and 
a scaling of variables $f(x) \mapsto f(s x)$
with a positive integer $s$.
With arithmetic or numerical operations on function values
we can define 
nonnegative multiplication of function values $f(x) \mapsto a f(x)$ with $a \geq 0$,
addition of a linear function
$f(x) \mapsto f(x) + \sum_{i=1}\sp{n} c_{i} x_{i}$
with  $c \in \RR\sp{n}$,
projection% 
\footnote{%%%%%%%%%%%%%%%
Here $x=(y,z)$ up to a permutation of components.
See (\ref{fsetprojdef}) in Section \ref{SCproj} 
for the precise meaning of the notation.
}  %%% footnote 
(partial minimization) $f(x) \mapsto  \inf_{z} f(y,z)$,
%%%%%%%%%%%%%%%%%%%%%%%%%%%%%%%%%%%%%%%%%%%%%%
%%Rockafellar 1970: Index contains ``projection of a convex function''
%%pointing to 38--38, 144, 25--56, but this term does not appear in the text.
%%Hiriart-Urruti and Lemarechal 2001 (Fudamentals of convex analysis):
%%p.97, Def. 2.4.4. marginal function in \cite[Def. 2.4.4]{HL01}.
%%%%%%%%%%%%%%%%%%%%%%%%%%%%%%%%%%%%%%%%%%%%%%
sum $f_{1} + f_{2}$ 
of two functions $f_{1}$ and $f_{2}$,
convolution
$(f_{1} \conv f_{2})(x) =
 \inf\{ f_{1}(y) + f_{2}(z) \mid x= y + z, \  y, z \in \ZZ\sp{n} \}$
of two functions $f_{1}$ and $f_{2}$,
etc.

Stability of discrete convexity under these operations
has been investigated for many function classes 
in discrete convex analysis,
such as 
\Lnat-convex functions,
\Mnat-convex functions,
and integrally convex functions
\cite{KMT07jump,MM17projcnvl,MMTT17proxIC,MMTT17dmpc,Mdcasiam,MS01rel}.
For multimodular functions, however, no systematic study
has been made, though there are results and observations scattered 
in the literature.

The objective of this paper is to investigate 
fundamental operations for multimodular functions
with particular interest in their connection to those for \Lnat-convex functions.
By compiling known and new results we shall arrive at a complete comparison 
of various kinds of discrete convexity with respect to fundamental operations,
as presented in Table \ref{TBoperation2dcfnZ} at the end of the paper.

This paper is organized as follows.
Section~\ref{SCmmfnprelim} is a review of relevant results on multimodular functions.
Section~\ref{SCchangevar} deals with operations defined by changes of variables and 
Section~\ref{SCfnvaloper} treats operations defined by arithmetic or numerical operations on functions values,
such as restriction, projection, and convolution.
In Section~\ref{SCconclrem} we conclude the paper 
with a table to compare the 
major classes of discrete convex functions.

\section{Multimodular Functions}
\label{SCmmfnprelim}

We consider functions defined on integer lattice points, 
$f: \ZZ\sp{n} \to \Rinf$,
where $\Rinf = \RR \cup \{ +\infty\}$
and the function may possibly take $+\infty$.
The {\em effective domain} of $f$ means the set of $x$
with $f(x) <  +\infty$ and is denoted 
by $\dom f =   \{ x \in \ZZ\sp{n} \mid  f(x) < +\infty \}$.

A function $f: \ZZ\sp{n} \to \Rinf$ is said to be {\em submodular} 
if it satisfies
\[
f(x) + f(y) \geq f(x \vee y) + f(x \wedge y)
\]
for all $x, y \in \ZZ\sp{n}$, 
where
$x \vee y$ and $x \wedge y$ denote,
respectively, the vectors of componentwise maximum and minimum of $x$ and $y$, 
i.e.,
\[ 
  (x \vee y)_{i} = \max(x_{i}, y_{i}),
\quad
  (x \wedge y)_{i} = \min(x_{i}, y_{i})
\qquad (i =1,2,\ldots, n).
\]

Let $\unitvec{i}$ denote the $i$th unit vector for $i =1,2,\ldots, n$,
and $\calF \subseteq \ZZ\sp{n}$ be the set of vectors defined by
\begin{equation} \label{multimodirection1}
\calF = \{ -\unitvec{1}, \unitvec{1}-\unitvec{2}, \unitvec{2}-\unitvec{3}, \ldots, 
  \unitvec{n-1}-\unitvec{n}, \unitvec{n} \} .
\end{equation}
A finite-valued function $f: \ZZ\sp{n} \to \RR$
is said to be {\em multimodular}
if it satisfies
\begin{equation} \label{multimodulardef1}
 f(z+d) + f(z+d') \geq   f(z) + f(z+d+d')
\end{equation}
for all $z \in \dom f$ and all distinct $d, d' \in \calF$
\cite{AGH00,Haj85}.
It is known \cite[Proposition 2.2]{Haj85} that
$f: \ZZ\sp{n} \to \RR$
%%$f: \ZZ\sp{n} \to \Rinf$ with $\dom f \not= \emptyset$
is multimodular if and only if the function 
$\tilde f: \ZZ\sp{n+1} \to \RR$ 
%%$\tilde f: \ZZ\sp{n+1} \to \Rinf$ 
defined by 
\begin{equation} \label{multimodular1}
 \tilde f(x_{0}, x) = f(x_{1}-x_{0},  x_{2}-x_{1}, \ldots, x_{n}-x_{n-1})  
 \qquad ( x_{0} \in \ZZ, x \in \ZZ\sp{n})
\end{equation}
is submodular in $n+1$ variables. 
This characterization enables us to define multimodularity 
for a function that may take the infinite value $+\infty$.
That is, we say that a function $f: \ZZ\sp{n} \to \Rinf$ with $\dom f \not= \emptyset$
is multimodular if the function 
$\tilde f: \ZZ\sp{n+1} \to \Rinf$
associated with $f$ by (\ref{multimodular1}) is submodular.

A function $g : \ZZ\sp{n} \to \Rinf$ with $\dom g \not= \emptyset$
is said to be {\em L$\sp{\natural}$-convex}%
\footnote{%%%%%%%%%%%%%%%
``L$\sp{\natural}$-convex'' should be read ``ell natural convex.''
} %%%footnote%%%%
if it has the property called ``discrete midpoint convexity,''
i.e., if it satisfies
\begin{equation} \label{midptcnv}
 g(p) + g(q) \geq
   g \left(\left\lceil \frac{p+q}{2} \right\rceil\right) 
  + g \left(\left\lfloor \frac{p+q}{2} \right\rfloor\right) 
\end{equation}
for all $p, q \in \ZZ\sp{n}$, 
where, for $z \in \RR$ in general, 
$\left\lceil  z   \right\rceil$ 
denotes the smallest integer not smaller than $z$
(rounding-up to the nearest integer)
and $\left\lfloor  z  \right\rfloor$
the largest integer not larger than $z$
(rounding-down to the nearest integer),
and this operation is extended to a vector
by componentwise applications.
It is known \cite{Mdcasiam} that
$g: \ZZ\sp{n} \to \Rinf$
with $\dom g \not= \emptyset$
is \Lnat-convex
if and only if the function 
$\tilde g: \ZZ\sp{n+1} \to \Rinf$
defined by 
\begin{equation}\label{lfnlnatfnrelation}
 \tilde g(p_{0},p) = g(p - p_{0} \vecone)
 \qquad ( p_{0} \in \ZZ, p \in \ZZ\sp{n})
\end{equation}
is submodular in $n+1$ variables,
where $\vecone=(1,1,\ldots,1)$. 
A function 
$h(q_{0}, q_{1}, \ldots, q_{n} )$
with $\dom h \not= \emptyset$ is called {\em {\rm L}-convex}
if it is submodular on $\ZZ\sp{n+1}$ and there exists
$r \in \RR$ such that 
\begin{equation}\label{linear}
h(q + \vecone) = h(q) +  r
\end{equation}
for all $q = (q_{0}, q_{1}, \ldots, q_{n} ) \in \ZZ\sp{n+1}$.
If $h$ is {\rm L}-convex,
the function
$h(0, q_{1}, \ldots, q_{n} )$
is \Lnat-convex, and any \Lnat-convex function arises in this way.
The function $\tilde g$ in (\ref{lfnlnatfnrelation}) derived from an \Lnat-convex function $g$  
is an {\rm L}-convex function, and we have $g(p)  =  \tilde g(0,p)$.

Multimodularity and L$\sp{\natural}$-convexity
have the following close relationship.

\begin{theorem}[\cite{Mmult05,Mdcaprimer07}] \rm \label{THmmfnlnatfn}
A function $f: \ZZ\sp{n} \to \Rinf$
is multimodular if and only if the function $g: \ZZ\sp{n} \to \Rinf$
defined by
\begin{equation} \label{mmfnGbyF}
 g(p) = f(p_{1}, \  p_{2}-p_{1}, \  p_{3}-p_{2}, \ldots, p_{n}-p_{n-1})  
 \qquad ( p \in \ZZ\sp{n})
\end{equation}
is L$\sp{\natural}$-convex.
%%\finbox
\end{theorem}
\begin{proof}
By definition, the multimodularity of $f$ is equivalent to
the submodularity of 
$\tilde f$ in (\ref{multimodular1}).
Since $\tilde f$ satisfies (\ref{linear}) for $r=0$, 
the submodularity of $\tilde f$ is equivalent to
the L-convexity of $\tilde f$.
On the other hand,
since $\tilde f(p_0 , p ) = g(p-p_0 \vecone)$,
the L-convexity of $\tilde f$ is equivalent to 
the L$\sp{\natural}$-convexity of $g$
(cf., (\ref{lfnlnatfnrelation})).
\end{proof}

Note that the relation (\ref{mmfnGbyF}) between $f$ and $g$ can be rewritten as
\begin{equation} \label{mmfnFbyG}
 f(x) = g(x_{1}, \  x_{1}+x_{2}, \  x_{1}+x_{2}+x_{3}, 
   \ldots, x_{1}+ \cdots + x_{n})  
 \qquad ( x \in \ZZ\sp{n}) .
\end{equation}
Using a bidiagonal matrix 
$D=(d_{ij} \mid 1 \leq i,j \leq n)$ defined by
\begin{equation} \label{matDdef}
 d_{ii}=1 \quad (i=1,2,\ldots,n),
\qquad
 d_{i+1,i}=-1 \quad (i=1,2,\ldots,n-1),
\end{equation}
we can express (\ref{mmfnGbyF}) and (\ref{mmfnFbyG}) 
more compactly  as $g(p)=f(Dp)$ and $f(x)=g(D\sp{-1}x)$, respectively. 
The matrix $D$ is unimodular, and its inverse
$D\sp{-1}$ is an integer matrix
with $(D\sp{-1})_{ij}=1$ for $i \geq j$ and 
$(D\sp{-1})_{ij}=0$ for $i < j$.
For $n=4$, for example, we have
\[
D = {\small
\left[ \begin{array}{rrrr}
1 & 0 & 0 & 0 \\
-1 & 1 & 0 & 0 \\
0 & -1 & 1 & 0 \\
0 & 0 & -1 & 1 \\
\end{array}\right]},
\qquad
D\sp{-1} = {\small
\left[ \begin{array}{rrrr}
1 & 0 & 0 & 0 \\
1 & 1 & 0 & 0 \\
1 & 1 & 1 & 0 \\
1 & 1 & 1 & 1 \\
\end{array}\right]}.
\]

\begin{remark} \rm \label{RMmmset}
The {\em indicator function} of a set $S \subseteq \ZZ\sp{n}$
is the function 
$\delta_{S}: \ZZ\sp{n} \to \{ 0, +\infty \}$
defined by
$\delta_{S}(x)  =
   \left\{  \begin{array}{ll}
    0            &   (x \in S) ,      \\
   + \infty      &   (x \not\in S) . \\
                      \end{array}  \right.$ 
A set $S$ is called an {\em \Lnat-convex set}
if its indicator function $\delta_{S}$ is \Lnat-convex.
Similarly, let us call a set $S$ a {\em multimodular set}
if its indicator function $\delta_{S}$ is multimodular. 
A multimodular set $S$ can be represented as 
$S = \{ D p \mid p \in T \}$
for some \Lnat-convex set $T$,
where $T$ is uniquely determined from $S$ as 
$T = \{ D\sp{-1} x \mid x \in S \}$.
It follows from (\ref{mmfnGbyF}) that
 the effective domain of a multimodular function is a multimodular set.
\finbox
\end{remark}

\begin{remark} \rm \label{RMmmfnn2mnat}
For functions in two variables,
multimodularity is the same as \Mnat-convexity.
That is, a function $f: \ZZ\sp{2} \to \Rinf$
is multimodular if and only if it is \Mnat-convex.
This fact follows easily from the definition
or from Theorem \ref{THmmfnlnatfn} and the relation between 
\Lnat-convex and \Mnat-convex functions for $n=2$.
See \cite{Mdcasiam} for the definition of \Mnat-convex functions.
\finbox
\end{remark}

A function
$f: \ZZ^{n} \to \Rinf$
in $x=(x_{1}, x_{2}, \ldots,x_{n}) \in \ZZ^{n}$
is called  {\em separable {\rm (}discrete{\rm )} convex}
if it can be represented as
$f(x) = \varphi_{1}(x_{1}) + \varphi_{2}(x_{2}) + \cdots + \varphi_{n}(x_{n})$
with univariate functions
$\varphi_{i}: \ZZ \to \Rinf$ satisfying 
$\varphi_{i}(t-1) + \varphi_{i}(t+1) \geq 2 \varphi_{i}(t)$
for all $t \in \ZZ$.

\begin{proposition} \rm \label{PRsepmmfn}
A separable convex function is multimodular.
\end{proposition}
\begin{proof}
For $f(x) = \sum_{i=1}\sp{n} \varphi_{i}(x_{i})$
the function $g$ in (\ref{mmfnGbyF}) is given as 
$g(p) = \varphi_{1}(p_{1}) + \sum_{i=2}\sp{n} \varphi_{i}(p_{i} - p_{i-1})$.
It is known \cite{Mdcasiam,Mbonn09} that such function is \Lnat-convex.
\end{proof}

A quadratic function admits a simple characterization of multimodularity
in terms of its coefficient matrix.

\begin{proposition} \rm \label{PRmmfnquadr}
A quadratic function $f(x) = x^{\top} A x$ is multimodular
if and only if
\begin{equation} \label{mmfquadrcond}
a_{ij} - a_{i,j+1} - a_{i+1,j} + a_{i+1,j+1} \leq 0
\qquad (0 \leq i < j \leq n),
\end{equation}
where $A=(a_{ij} \mid i,j =1,2,\ldots, n)$ and $a_{ij} =0$ if $i=0$ or $j=n+1$.
\end{proposition}
\begin{proof}
The inequality (\ref{multimodulardef1})
for  $d = \unitvec{i} - \unitvec{i+1}$ and $d' = \unitvec{j} - \unitvec{j+1}$,
where $\unitvec{0} = \unitvec{n+1} = \veczero$ by convention,
is equivalent to
$(\unitvec{i} - \unitvec{i+1})\sp{\top} A (\unitvec{j} - \unitvec{j+1}) \leq 0$.
This is further equivalent to (\ref{mmfquadrcond}).
\end{proof}

\begin{remark} \rm \label{RMmmfnquadLnat}
Here is an alternative proof of Proposition \ref{PRmmfnquadr} via \Lnat-convexity. 
Let $\mathcal{L}$ denote the set of all $n \times n$ symmetric matrices $B=(b_{ij})$
such that $b_{ij} \leq 0$ for all $i\not= j$
and $b_{ii} \geq \sum_{j\not= i} |b_{ij}|$ for all $i$.
It is known \cite{Mdcasiam} that 
$g(p) = p^{\top} B p$ is \Lnat-convex if and only if
$B$ belongs to $\mathcal{L}$.
Then, by Theorem \ref{THmmfnlnatfn},
$f(x) = x^{\top} A x$
is multimodular if and only if
$D\sp{\top} A D$ belongs to $\mathcal{L}$.
This latter condition is equivalent to (\ref{mmfquadrcond}).
\finbox
\end{remark}

The following nice properties of multimodular functions are worth mentioning,
though we do not use them in this paper.

\begin{itemize}
\item
An integer point  $x \in \dom f$ is a (global) minimizer of a multimodular function $f$
if and only if it is a local minimizer in the sense that
$f(x) \leq f(x \pm d)$ for all $d \in \calT$, 
where $\calT$ is the set of vectors of the form
$\unitvec{i_{1}} - \unitvec{i_{2}} + \cdots + (-1)\sp{k-1} \unitvec{i_{k}}$
for some increasing sequence of indices $i_{1} < i_{2} < \cdots < i_{k}$
\cite[Theorem 3.1]{Mmult05}.

\item
A multimodular function $f$ can be extended to a convex function 
in a specific manner \cite[Theorem 2.1]{AGH00}.
Furthermore, a multimodular function is integrally convex \cite[Section 14.6]{Mdcaprimer07};
see \cite{Mdcasiam} for the definition of integrally convex functions.

\item
A discrete separation theorem holds for multimodular functions
\cite[Theorem 4.1]{Mmult05}.
Let $f: \ZZ\sp{n} \to \RR \cup \{ +\infty\}$
and $g: \ZZ\sp{n} \to \RR \cup \{ -\infty\}$
be functions such that $f$ and $-g$ are multimodular,
and assume that $f(x_{0})$ and $g(x_{0})$ are finite
for some $x_{0} \in \ZZ\sp{n}$.
If $f(x) \geq g(x)$  for all $x \in \ZZ\sp{n}$,  there exist 
$\alpha\sp{*} \in \RR$ and $p\sp{*} \in \RR\sp{n}$ such that
$f(x) \geq \alpha\sp{*} + \langle p\sp{*}, x \rangle  \geq g(x)$
for all $x \in \ZZ\sp{n}$, 
where $\langle \cdot, \cdot \rangle$ denotes the standard inner product of vectors.
Moreover, if $f$ and $g$ are integer-valued,
there exist integer-valued 
$\alpha\sp{*} \in \ZZ$ and $p\sp{*} \in \ZZ\sp{n}$.
\end{itemize}

%%%%%%%%%%%%%%%%%%%%%%

\section{Operations via Change of Variables}
\label{SCchangevar}

In this section we consider multimodularity of functions induced
by changes of variables such as an origin shift,
a sign inversion of variables, a permutation of variables, 
and a scaling of variables.
We consistently adopt the proof strategy to translate 
the operations for multimodular functions to those for \Lnat-convex functions,
so that we can better understand the connection
between multimodularity and \Lnat-convexity.
In the proofs we use notations
$f$ for a given multimodular function, 
$\tilde f$ for the function resulting from the operation, and
\begin{align} 
 g(p) &= f(p_{1}, \  p_{2}-p_{1}, \  p_{3}-p_{2}, \ldots, p_{n}-p_{n-1})  = f(Dp),
\label{mmfnGbyFproof1}
\\
 \tilde g(p) &= \tilde f(p_{1}, \  p_{2}-p_{1}, \  p_{3}-p_{2}, \ldots, p_{n}-p_{n-1}) 
  = \tilde f(Dp),
\label{mmfnGbyFproof2}
\end{align}
which imply
\begin{align} 
 f(x) &= g(x_{1}, \  x_{1}+x_{2}, \  x_{1}+x_{2}+x_{3},  \ldots, x_{1}+ \cdots + x_{n}) 
       = g(D\sp{-1}x),
\label{mmfnFbyGproof1}
\\
 \tilde f(x) &= \tilde g(x_{1}, \  x_{1}+x_{2}, \  x_{1}+x_{2}+x_{3}, \ldots, x_{1}+ \cdots + x_{n}) 
  = \tilde g(D\sp{-1}x).
\label{mmfnFbyGproof2}
\end{align}

We start with an origin shift and a sign inversion of variables.

\begin{proposition} \rm \label{PRmmfnorgshift}
For a multimodular function $f$ and an integer vector $b$,
the function $\tilde f(x) = f(x+b)$ is multimodular.
\end{proposition}
\begin{proof}
By (\ref{mmfnFbyGproof1}) and (\ref{mmfnFbyGproof2}),
we can translate $\tilde f(x) = f(x+b)$ to $\tilde g(p) = g(p+c)$   with 
$c = (b_{1},  b_{1}+b_{2}, b_{1}+b_{2}+b_{3},  \ldots, b_{1}+ \cdots + b_{n})$,
where $g$ is \Lnat-convex.
Then $\tilde g$ is also \Lnat-convex,
since \Lnat-convexity is stable under an origin shift.
\end{proof}

\begin{proposition} \rm \label{PRmmfnvarminum}
For a multimodular function $f$,
the function $\tilde f(x) = f(-x)$ is multimodular.
\end{proposition}
\begin{proof}
By (\ref{mmfnFbyGproof1}) and (\ref{mmfnFbyGproof2}),
we can translate
$\tilde f(x) = f(-x)$ to
$\tilde g(p) = g(-p)$, where $g$ is \Lnat-convex.
Then $\tilde g$ is also \Lnat-convex,
since \Lnat-convexity is stable under a sign inversion of variables.
\end{proof}

It is known that
reversing the ordering of variables preserves multimodularity
\cite[Remarks (1)]{Haj85}.
It is emphasized that this is not obvious 
since the definition of multimodularity
depends on the ordering of variables.

\begin{proposition} [\cite{Haj85}] \rm \label{PRmmfnvarinv}
For a multimodular function $f$,
the function $\tilde f$ defined by
$\tilde f(x_{1}, x_{2}, \ldots, x_{n}) \allowbreak = f(x_{n}, \ldots, x_{2}, x_{1})$
is multimodular.
\end{proposition}
\begin{proof}
We give an alternative proof  via \Lnat-convexity in accordance with our strategy.
Let $R=(r_{ij})$ denote the permutation matrix 
representing the reversal of the ordering, i.e.,
$r_{i,n+1-i}=1$ for $i=1,2,\ldots,n$ and other entries being zero.
Then we have $\tilde f(x) = f(R x)$.
By (\ref{mmfnFbyGproof1}) and (\ref{mmfnFbyGproof2}),
we can translate
$\tilde f(x) = f(R x)$ to
$\tilde g(D\sp{-1}x)  = g(D\sp{-1} R x)$,
that is, $\tilde g(p)  = g(D\sp{-1} R D p)$.
A direct calculation shows that the matrix $T=(t_{ij})=D\sp{-1} R D$ 
is given by:
$t_{in}=1$ ($i=1,2,\ldots,n$),
$t_{i,n-i}=-1$ ($i=1,2,\ldots,n-1$),
and 
$t_{ij} =0$ for other $(i,j)$.
For $n=4$, for example, we have
$T=D\sp{-1} R D =
{\footnotesize
\left[ \begin{array}{rrrr}
0 & 0 & -1 & 1 \\
0 & -1 & 0 & 1 \\
-1 & 0 & 0 & 1 \\
0 & 0 & 0 & 1 \\
\end{array}\right]}$.
Then we obtain%
\footnote{%%%%%%%%%%%%%%%%%%%%%
It is somewhat surprising that the order reversal of variables
corresponds to the transformation (\ref{mmfnreversal2}) for \Lnat-convex functions.
} %%%%%% footnote %%%%%%%%%%%%%%%%%
\begin{equation} \label{mmfnreversal2} 
\tilde g(p)  = g( -(p_{n-1},p_{n-2}, \ldots, p_{1}, 0) + p_{n}\vecone ) .
\end{equation}
The \Lnat-convexity of $\tilde g$ can be seen as follows.
Define $h: \ZZ\sp{n+1} \to \Rinf$ by
\[
h(p_{0}, p_{1}, p_{2}, \ldots, p_{n} ) 
 = g( -(p_{n-1},p_{n-2}, \ldots, p_{1}, p_{0}) + p_{n}\vecone ) 
\]
and
 $g\sp{\rm rev}: \ZZ\sp{n} \to \Rinf$ by
\[
g\sp{\rm rev}(p_{0},p_{1}, \ldots, p_{n-2}, p_{n-1})
 = g(-p_{n-1},-p_{n-2}, \ldots, -p_{1}, -p_{0}).
\]
The function $h$ is {\rm L}-convex, 
since $g\sp{\rm rev}$ is \Lnat-convex 
and the function derived from $g\sp{\rm rev}$
by (\ref{lfnlnatfnrelation}) coincides with $h$.
Then the relation
$\tilde g(p)  = h(0, p_{1}, p_{2}, \ldots, p_{n} )$
in (\ref{mmfnreversal2}) 
means that $\tilde g$ is obtained from an {\rm L}-convex function by restriction.
Therefore, $\tilde g$ is \Lnat-convex.
\end{proof}

Not every permutation of variables preserves multimodularity.

\begin{example} \rm \label{EXmmperm1}
The quadratic function
$f(x) = x^{\top} A x$
with
$A  = {\small\footnotesize
\left[ \begin{array}{rrr}
 1 & 1 & 0 \\
 1 & 2 & 1 \\
 0 & 1 & 1
\end{array}\right]}$
is multimodular,
whereas 
$\tilde f(x_{1}, x_{2}, x_{3}) = f(x_{2}, x_{1}, x_{3})$
arising from a transposition is not multimodular.
Indeed we have
$\tilde f(x) = x^{\top} \tilde A x$
for
$\tilde A  = {\small\footnotesize
\left[ \begin{array}{rrr}
 2 & 1 & 1 \\
 1 & 1 & 0 \\
 1 & 0 & 1  \\
\end{array}\right]}$,
for which the condition (\ref{mmfquadrcond}) fails for $(i,j)=(1,3)$.
Referring to Remark \ref{RMmmfnquadLnat} we also note that
$B = D\sp{\top} A D =
{\small\footnotesize
\left[ \begin{array}{rrr}
 1 & 0  & -1  \\
 0 & 1 & 0 \\
 -1 & 0 & 1
\end{array}\right]} \in \mathcal{L}$
\ and \ 
$\tilde B = D\sp{\top} \tilde A D =
{\small\footnotesize
\left[ \begin{array}{rrr}
 1 & -1  & \framebox{1}  \\
 -1 & 2 & -1 \\
 \framebox{1} & -1 & 1
\end{array}\right]}   \not\in \mathcal{L}$.
A cyclic permutation of variables 
$f(x_{3}, x_{1}, x_{2})$
is not multimodular, either,
since it coincides with $x^{\top} \tilde A x$. 
\finbox
\end{example}

A scaling of variables preserves multimodularity.

\begin{proposition} \rm \label{PRmmfnvarscaling}
For a multimodular function $f$
and a positive integer $s$,
the function 
$\tilde f(x) = f(s x)$
is multimodular.
\end{proposition}
\begin{proof}
By (\ref{mmfnFbyGproof1}) and (\ref{mmfnFbyGproof2}),
we can translate
$\tilde f(x) = f(s x)$ to
$\tilde g(p) = g(s p)$, where $g$ is \Lnat-convex.
Then $\tilde g$ is also \Lnat-convex,
since \Lnat-convexity is stable under a scaling of variables \cite{Mdcasiam}.
\end{proof}

\section{Operations Relating to Function Values}
\label{SCfnvaloper}

In this section we consider multimodularity of functions 
resulting from operations such as 
nonnegative multiplication of function values,
addition of a linear function,
projection (partial minimization),
sum of two functions, and
convolution of two functions.
We continue with the proof strategy of translating 
the operations for multimodular functions to those for \Lnat-convex functions.

\subsection{Multiplication and Addition}
\label{SCsumetc}

We start with simple operations, for which the following statements are obvious.

\begin{proposition}[\cite{AGH00}] \rm \label{PRmmfnsumetc}
Let $f$, $f_{1}$, $f_{2}$  be multimodular functions.

\noindent
{\rm (1)}
For any $a \geq 0$,
$\tilde f(x) = a f(x)$ 
is multimodular.

\noindent
{\rm (2)}
For any $c \in \RR\sp{n}$,
$\tilde f(x) = f(x) + \sum_{i=1}\sp{n} c_{i} x_{i}$
is multimodular.

\noindent
{\rm (3)}
For any separable convex function $\varphi(x)$,
$\tilde f(x) = f(x) + \varphi(x)$
is multimodular.

\noindent
{\rm (4)}
Sum $\tilde f(x) =  f_{1}(x) + f_{2}(x)$ 
is multimodular.
\finbox
\end{proposition}

\subsection{Restriction}
\label{SCrestr}

Let $N = \{ 1,2,\ldots, n \}$.
For a function
$f: \ZZ\sp{N} \to \Rinf$ and a subset $U \subseteq N$, the
{\em restriction}
of $f$ to $U$ is a function $f_{U}: \ZZ\sp{U} \to \Rinf$ defined by%
\footnote{%%%%%%%%%%%%%%%%%%%%%
For any $z\in \ZZ\sp{N \setminus U}$ we may consider a function $f(y,z)$ in $y \in \ZZ\sp{U}$.
For simplicity we choose $z=\veczero_{N \setminus U}$.
} %%%%%% footnote %%%%%%%%%%%%%%%%%
\begin{eqnarray} 
 f_{U}(y) &=& f(y,\veczero_{N \setminus U})
  \qquad (y \in \ZZ\sp{U}) ,
\label{fsetrestrict} 
\end{eqnarray}
where $\veczero_{N \setminus U}$ denotes the zero vector 
in $\ZZ\sp{N \setminus U}$.
The notation 
$(y,\veczero_{N \setminus U})$ means the vector
whose $i$th component is equal to $y_{i}$ for $i \in U$
and to 0 for $i \in N \setminus U$;
for example, if $N = \{ 1,2,3 \}$ and $U = \{ 1,3 \}$,
$(y,\veczero_{N \setminus U})$ means $(y_{1}, 0, y_{3})$.

The restriction of a multimodular function is known to be multimodular
\cite[Lemma 2.3]{AGH00} (see also \cite[Lemma 3]{AGH03}).

\begin{proposition}[\cite{AGH00}] \rm \label{PRmmfnrestr}
For a multimodular function $f$ and any subset $U$,
the restriction $f_{U}$
is multimodular, provided that $\dom f_{U} \neq \emptyset$.
\end{proposition}

\begin{proof}
We give an alternative proof in accordance with our strategy.
It suffices to consider the case where
$N \setminus U = \{ k \}$ for some $k \in N$.
Define $\tilde f (x_{1}, \ldots, x_{k-1}, x_{k+1}, \ldots,x_{n})
= f (x_{1}, \ldots, x_{k-1}, 0, x_{k+1},  \allowbreak   \ldots,x_{n})$.
Then $\tilde f$ is multimodular
if and only if 
the inequality (\ref{multimodulardef1}) holds for $f$
for all $z \in \ZZ\sp{n}$ and all distinct elements $d, d'$ of 
\[
\tilde \calF = 
\calF \setminus \{ \unitvec{k-1}-\unitvec{k}, \unitvec{k}-\unitvec{k+1} \}
   \cup \{  \unitvec{k-1}-\unitvec{k+1} \} ,
\]
where $\unitvec{0}=\unitvec{n+1}=\bm{0}$.
We use notation
$\psi(x) = (x_{1}, \  x_{1}+x_{2}, \  x_{1}+x_{2}+x_{3},  \ldots, x_{1}+ \cdots + x_{n})$
for the transformation $x \mapsto p$ 
in (\ref{mmfnFbyG}),
i.e., $f(x)=g(\psi(x))$.
If $k = 1$, we have
\begin{align*}
\psi(-\unitvec{2}) &= (0,-1, \ldots, -1) = \unitvec{1}-\vecone,
\\  
\psi(\unitvec{i}-\unitvec{i+1}) &= \unitvec{i}
 \qquad (i \in \{ 2, \ldots, n \})
\end{align*}
for the elements of $\tilde \calF$, 
and therefore, $\tilde f$ is multimodular if and only if 
\begin{align}
& 
g(p + \unitvec{i}) + g(p + \unitvec{j}) \geq g(p) + g(p + \unitvec{i} + \unitvec{j}  ) ,
\label{gijk1}
\\ &
g(p + \unitvec{i}) + g(p + \unitvec{1} - \vecone ) \geq g(p) + g(p + \unitvec{i} + \unitvec{1} - \vecone   ) ,
\label{gi1k1}
\end{align}
where $i, j \in \{ 2, \ldots, n \}$ and $i \not= j$.
If $2 \leq k \leq n$, we have
\begin{align*}
\psi(-\unitvec{1}) &= (-1,-1, \ldots, -1) = -\vecone,
\\  
\psi(\unitvec{i}-\unitvec{i+1}) &= \unitvec{i}
 \qquad (i \in \{ 1, \ldots, k-2 \} \cup \{ k+1, \ldots, n \}),
\\  
\psi(\unitvec{k-1}-\unitvec{k+1}) &= \unitvec{k-1} + \unitvec{k}
\end{align*}
for the elements of $\tilde \calF$, 
and therefore, $\tilde f$ is multimodular if and only if 
\begin{align}
& 
g(p + \unitvec{i}) + g(p + \unitvec{j}) \geq g(p) + g(p + \unitvec{i} + \unitvec{j}  ) ,
\label{gij}
\\ &
g(p + \unitvec{i}) + g(p + \unitvec{k-1} + \unitvec{k} ) \geq g(p) + g(p + \unitvec{i} + \unitvec{k-1} + \unitvec{k}   ) ,
\label{gikk}
\\ &
g(p + \unitvec{i}) + g(p - \vecone ) \geq g(p) + g(p + \unitvec{i} - \vecone   ) ,
\label{gi1}
\\ &
 g(p+ \unitvec{k-1} + \unitvec{k}) + g(p - \vecone )  \geq g(p) + g(p + \unitvec{k-1} + \unitvec{k} - \vecone  ) 
\label{gikk1},
\end{align}
where $i, j \in \{ 1, \ldots, k-2 \} \cup \{ k+1, \ldots, n \}$ and $i \not= j$.
We finally observe that inequalities 
(\ref{gijk1})--(\ref{gikk1}) hold by the discrete midpoint convexity (\ref{midptcnv}) of $g$.
\end{proof}

\subsection{Projection}
\label{SCproj}

For a function
$f: \ZZ\sp{N} \to \Rinf$ and a subset $U \subseteq N$, the
{\em projection}
of $f$ to $U$ means a function
$f\sp{U}: \ZZ\sp{U} \to \Rmpinf$
defined by
\begin{equation} \label{fsetprojdef} 
  f\sp{U}(y)  =  \inf \{ f(y,z) \mid z \in \ZZ\sp{N \setminus U} \}
  \qquad (y \in \ZZ\sp{U}) ,
\end{equation}
where the notation $(y,z)$ means the vector
whose $i$th component is equal to $y_{i}$ for $i \in U$
and to $z_{i}$ for $i \in N \setminus U$;
for example, if $N = \{ 1,2,3,4 \}$ and $U = \{ 2,3 \}$,
$(y,z) = (z_{1}, y_{2}, y_{3}, z_{4})$.
We assume $f\sp{U} > -\infty$.
The projection is sometimes called {\em partial minimization}.

A subset $U$ of $N = \{ 1,2,\ldots, n \}$
is said to be an {\em interval}
if it consists of consecutive numbers.
The projection of a multimodular function to an interval is multimodular.

\begin{proposition} \rm \label{PRmmfnprojinterval}
For a multimodular function $f$ and an interval $U$,
the projection 
$f\sp{U}$
is multimodular.
\end{proposition}
\begin{proof}
We first consider the case of $U = N \setminus \{ n \}$.
By (\ref{fsetprojdef}) and (\ref{mmfnFbyG}) we obtain
\begin{align*}
& f\sp{U}(x_{1},x_{2},\ldots,x_{n-1} ) 
\\ &
 = \inf_{z \in \ZZ} f(x_{1},x_{2},\ldots,x_{n-1}, z)
\\ &=
 \inf_{z \in \ZZ} 
  g(x_{1}, x_{1}+x_{2}, \ldots, \  x_{1}+ \cdots + x_{n-1}, \  x_{1}+ \cdots + x_{n-1} + z)  
\\ &
= g\sp{U}(x_{1},  x_{1}+x_{2},  \ldots, x_{1}+ \cdots + x_{n-1})  ,
\end{align*}
where $g\sp{U}$ denotes the projection of $g$ to $U$.
Here $g\sp{U}$ is \Lnat-convex, since
the projection of an \Lnat-convex function is known \cite[Theorem 7.11]{Mdcasiam} to be \Lnat-convex.
Therefore, $f\sp{U}$ is multimodular.

The case of $U = N \setminus \{ 1 \}$
can be reduced to the above case by Proposition \ref{PRmmfnvarinv},
which allows us to reverse the ordering of variables.
For a general interval $U$, we repeat eliminating variables from 
both ends of $\{ 1,2,\ldots, n \}$.
\end{proof}

The projection of a multimodular function 
to an arbitrary subset $U$ is not necessarily multimodular.

\begin{example}   \rm \label{EXmmfnproj4}
The quadratic function
$f(x) = x^{\top} A x$
with
$A  = {\small\footnotesize
\left[ \begin{array}{cccc}
  3 & 2 & 1 & 0 \\
 2 & 3 & 2 &  1 \\
 1 & 2 & 2 & 1 \\
 0 & 1 & 1 &  1 \\
\end{array}\right]
}%%font
$
is multimodular,
whereas its projection
$f\sp{U}$ to $U= \{ 1,2,4 \}$
is not.
Indeed we have
$f\sp{U}(y) = y^{\top} \tilde A y$
for
$\tilde A = {\small\footnotesize
{\displaystyle {\small\footnotesize \frac{1}{2}}  }
\left[ \begin{array}{rrr}
 5 & 2 & -1 \\
 2 & 2 & 0 \\
 -1 & 0 & 1 \\
\end{array}\right]}$,
where
$\tilde A =(\tilde a_{ij} \mid  i,j =1,2,4)$
is obtained from $A$ by the usual sweep-out operation:
 $\tilde a_{ij} = a_{ij} - a_{i3} a_{3j} / a_{33}$
$(i,j \in \{ 1,2,4 \})$.
The matrix $\tilde A$  violates the condition (\ref{mmfquadrcond}) for $(i,j)=(1,2)$.
Referring to Remark \ref{RMmmfnquadLnat} we also note that
$B = D_{4}\sp{\top} A D_{4} =
{\small\footnotesize
\left[ \begin{array}{rrrr}
2 & 0 & 0 &  -1 \\
0 & 1 & 0 & 0   \\
0 & 0  & 1 & 0  \\
-1 &  0 & 0 & 1  \\
\end{array}\right]} \in \mathcal{L}$
\ and \ 
$\tilde B = D_{3}\sp{\top} \tilde A D_{3} =
{\displaystyle { \small \frac{1}{2} }}
{\small\footnotesize
\left[ \begin{array}{rrr}
 3 & \framebox{1} & -1 \\
 \framebox{1} & 3 & -1 \\
 -1 & -1 & 1  \\
\end{array}\right]}   \not\in \mathcal{L}$,
where $D_{4}$ and $D_{3}$ are $4 \times 4$ and  $3 \times 3$ matrices defined as in (\ref{matDdef}).
\finbox
\end{example}

\subsection{Convolution}
\label{SCconvol}

The (infimal) {\em convolution} of two functions
$f_{1}, f_{2}: \ZZ\sp{n} \to \Rinf$
is defined by
\begin{equation} \label{f1f2convdef}
(f_{1} \conv f_{2})(x) =
 \inf\{ f_{1}(y) + f_{2}(z) \mid x= y + z, \  y, z \in \ZZ\sp{n}  \}
\qquad (x \in \ZZ\sp{n}) ,
\end{equation}
where it is assumed that the infimum 
is bounded from below (i.e., $\not= -\infty$) for every $x \in \ZZ\sp{n}$.
The {\em Minkowski sum} of two sets $S_{1}$, $S_{2} \subseteq \ZZ\sp{n}$ 
is defined by
\begin{equation} \label{minkowsumZdef}
S_{1}+S_{2} = \{ y + z \mid y \in S_{1}, z \in S_{2} \} .
\end{equation}
The indicator function of the Minkowski sum
coincides with the convolution of the respective indicator functions,
i.e., 
$\delta_{S_{1}+S_{2}} =  \delta_{S_{1}} \conv \delta_{S_{1}}$.

Example \ref{EXmltMdim3set} below shows the following facts. 
Recall that a multimodular set means a set whose indicator function is multimodular
(Remark \ref{RMmmset})
and that a separable convex function is multimodular (Proposition \ref{PRsepmmfn}).

\begin{itemize}
\item

The Minkowski sum of a multimodular set and an integer interval (box)
is not necessarily a multimodular set.

\item
The convolution $f \conv \varphi$ of a multimodular function $f$
and a separable convex function $\varphi$ is not necessarily a multimodular function.

\item
The convolution $f_{1} \conv f_{2}$ of two multimodular functions $f_{1}$ and $f_{2}$
is not necessarily a multimodular function.
\end{itemize}

\begin{example}  \rm \label{EXmltMdim3set}
Let
$S_{1} = \{ (0,0,0), (1,0,-1) \}$ and
$S_{2} = \{ (0,0,0), (0,1,0) \}$,
where $S_{2}$ is an integer interval.
Both $S_{1}$ and $S_{2}$ are multimodular, but 
their Minkowski sum
$S_{1} + S_{2} = \{ (0,0,0), (1,0,-1),  \allowbreak    (0,1,0), (1,1,-1) \}$ 
is not multimodular.
We can check this directly or via transformation 
to $T_{i} = \{ D\sp{-1} x \mid x \in S_{i} \}$
for $i=1,2$.
We have
$T_{1} = \{ (0,0,0), (1,1,0) \}$ and $T_{2} = \{ (0,0,0), (0,1,1) \}$,
which are easily seen to be \Lnat-convex.
But their Minkowski sum 
$T_{1} +  T_{2} = \{(0, 0, 0), (0, 1, 1),  \allowbreak    (1, 1, 0), (1, 2, 1)\}$
is not \Lnat-convex,
since for $p=(0, 1, 1)$ and $q=(1, 1, 0)$ in $T_{1} + T_{2}$,
we have
$\left\lceil (p+q)/2 \right\rceil = (1, 1, 1) \not\in T_{1} + T_{2}$
and
$\left\lfloor (p+q)/2 \right\rfloor = (0, 1, 0) \not\in T_{1} + T_{2}$.
Since $T_{1} +  T_{2} =  \{ D\sp{-1} x \mid x \in S_{1} + S_{2} \}$,
this means that $S_{1} +  S_{2}$ is not multimodular.
It it mentioned that this example is based on
the example for \Lnat-convex sets given in 
\cite[Note 5.11]{Mdcasiam} and \cite[Example 3.11]{MS01rel}.
\finbox
\end{example}

\section{Concluding Remarks}
\label{SCconclrem}

Multimodular functions have been used as a fundamental tool 
to analyze recurrence relations
in the literature of queueing theory, discrete-event systems, and operations research.
In some analysis, propagation or stability of multimodularity 
through  recurrence formulas plays a critical role.
A recurrence formula consists of various kinds of operations,
some of which preserve multimodularity and others not.
The projection operation (partial minimization) is closely related to 
the Bellman equation in dynamic programming,
and the assumption of $U$ being an interval (consecutive variables)  
in Proposition \ref{PRmmfnprojinterval} is quite natural in this interpretation.
The reversal of the ordering of variables in Proposition \ref{PRmmfnvarinv}
corresponds to the reversal of ``time'' in recurrence relations.
It is hoped that the results of this paper will find applications
in concrete problems in operations research.

The known facts about fundamental operations on discrete convex functions,
including those obtained in this paper,
are summarized in Table~\ref{TBoperation2dcfnZ}.

%%%%%%%%%% table %%%%%%%%%%
\addtolength{\tabcolsep}{-4pt}
\begin{table}
\begin{center}
\caption{Fundamental operations on discrete convex functions}
\label{TBoperation2dcfnZ}

\

{\small
\begin{tabular}{l|cc|cc|cc|cc|l}
 Discrete  & \multicolumn{2}{c|}{Variables} & Restric-  & Projec- & \multicolumn{2}{c|}{Addition} & \multicolumn{2}{c|}{\!\! Convolution} & Reference
 \\  \cline{6-9} \cline{2-3}
 \qquad convexity &  Permut. \!\! &  \! Scaling & tion & tion  	 
 & \small $f + \varphi$ & \!\! \small $f_{1} + f_{2}$ & \!\! \small $f \conv \varphi$ & \!\!\!\!  \small $f_{1} \conv f_{2}$ &
\\ \hline
 Separable conv & Y \ \ & Y \ \ & Y \ \  & Y \ \  & Y \ \  & Y \ \  & Y \ \  & Y \ \ &
\\ 
 Integrally conv & Y \ \ &  \ \ \textbf{\textit{N}} & Y \ \  & Y \ \  & Y \ \  & \ \  \textbf{\textit{N}} & Y \ \  & \ \  \textbf{\textit{N}}
    & \cite{MM17projcnvl,MMTT17proxIC,MS01rel}
\\ 
\Lnat-convex & Y \ \ & Y \ \  & Y \ \  & Y \ \  & Y \ \  & Y \ \  & Y \ \  & \ \  \textbf{\textit{N}} 
    & \cite{Mdcasiam}
\\ 
{\rm L}-convex  & Y \ \ & Y \ \ & \ \  \textbf{\textit{N}}  & Y \ \  & \ \  \textbf{\textit{N}}
 & Y \ \  & Y \ \  & \ \  \textbf{\textit{N}} 
    & \cite{Mdcasiam}
\\ 
\Mnat-convex  & Y \ \ &  \ \ \textbf{\textit{N}} &  Y \ \  & Y \ \  & Y \ \  & \ \  \textbf{\textit{N}} & Y \ \  & Y \ \ 
    & \cite{Mdcasiam}
\\ 
M-convex & Y \ \ &  \ \ \textbf{\textit{N}} & Y \ \  & \ \  \textbf{\textit{N}} & Y \ \  & \ \  \textbf{\textit{N}} & \ \ \textbf{\textit{N}} & Y \ \ 
    & \cite{Mdcasiam}
\\ \hline
\small
 & \ \   &  \ \ &  Y \ \   & \ \    & Y \ \  & Y \ \  & \ \    & \ \   
    & \cite{AGH00}
\\ 
\small
\framebox{Multimodular} & \ \  \textbf{\textit{N}} & Y: Prop.\ref{PRmmfnvarscaling}\ \ & alt. proof & \ \  \textbf{\textit{N}}  &  \ \  &  \ \  & \ \  \textbf{\textit{N}}  & \ \  \textbf{\textit{N}} 
    & this paper %[$*$]
\\  
\small
 & \ \  Y*:Prop.\ref{PRmmfnvarinv} & \ \ & \ \  & \ \  \ \   &  \ \  &  \ \  & \ \  \ \   & \ \  \ \  
    & \cite{Haj85} %[$*$]
\\
\small
 & \ \  alt. proof &  \ \ &   \ \   & \ \  Y*:Prop.\ref{PRmmfnprojinterval}  &  \ \  &  \ \  & \ \     & \ \  
    & this paper %[$*$]
\\ \hline 
\small 
Globally d.m.c.&  Y \ \ & Y \ \ &  Y \ \   & Y \ \  & Y \ \  & Y \ \  & \ \  \textbf{\textit{N}} & \ \  \textbf{\textit{N}} 
    & \cite{MMTT17dmpc}
\\ 
\small
Locally d.m.c.  &  Y \ \ & Y \ \ &  Y \ \   &  Y \ \  & Y \ \  & Y \ \  & \ \  \textbf{\textit{N}}  & \ \  \textbf{\textit{N}} 
    & \cite{MMTT17dmpc}
\\ 
M-conv (jump) &  Y \ \ & \ \ \textbf{\textit{N}} & Y \ \   & Y \ \  & Y \ \  & \ \  \textbf{\textit{N}} & Y \ \  & Y \ \ 
    & \cite{KMT07jump,Mmjump06}
\\ \hline
\multicolumn{5}{l}{d.m.c.: discrete midpoint convex,} 
    & \multicolumn{5}{l}{$\varphi$: separable convex}%, \quad [$*$]: this paper} 
\\
\multicolumn{9}{l}{Y: Discrete convexity (of that kind) is preserved, \ \ \textbf{\textit{N}}: Not preserved} \\
\multicolumn{9}{l}{Y*: Discrete convexity (of that kind) is preserved in some cases} \ \  
\end{tabular}
\addtolength{\tabcolsep}{4pt}
}%font/small
\end{center}
\end{table}
%%%%%%%%%% table %%%%%%%%%%

%%\clearpage\newpage
%%\end{document}

\end{document}